\documentclass{article}
\usepackage{geometry}
\usepackage{titling}
\AtEndDocument{\bigskip{\footnotesize%
\textsc{${}^*$School of Basic and Applied Science, Lingaya's Vidyapeeth Faridabad, Haryana–121002, India} \par  
  \textit{E-mail address:} \texttt{dr.surya@lingayasvidyapeeth.edu.in} \par

  \addvspace{\medskipamount}
}}
\usepackage{doc}

\usepackage{url}

\usepackage{graphicx}
\usepackage{epstopdf}
\usepackage{amsthm}
\usepackage{amssymb}
\usepackage{amsmath}
\usepackage{enumitem}

\usepackage{array}
\usepackage{enumitem}
\usepackage[utf8]{inputenc}
\usepackage[english]{babel}
\newtheorem{theorem}{Theorem}
\newtheorem{corollary}{Corollary}[theorem]
\newtheorem{lemma}[theorem]{Lemma}

\DeclareMathOperator{\RE}{Re}

\usepackage{varwidth}
\usepackage{lineno, hyperref}
\usepackage{graphicx}
\usepackage{epstopdf}
\usepackage{amsmath}
\usepackage{amsthm}
\usepackage{enumitem}
\usepackage{amsthm}
\usepackage{float}
\usepackage{subcaption}
\usepackage{caption}
\usepackage{authblk}
\usepackage{abstract}

\begin{document}
\title{ Toeplitz Determinants for Inverse Functions and their Logarithmic Coefficients Associated with Ma-Minda Classes }
\author{Surya Giri$^{*}$  }


\date{}


	

\maketitle	

\begin{abstract}
    \noindent The classes of analytic univalent functions on the unit disk defined by
   $$ \mathcal{S}^*(\varphi)= \bigg\{ f \in \mathcal{A}: \frac{z f'(z)}{f(z)} \prec \varphi(z)\bigg\}$$
    and
   $$ \mathcal{C}(\varphi)=\bigg\{ f \in \mathcal{A}: 1 + \frac{z f''(z)}{f'(z)} \prec \varphi(z)\bigg\} $$
  generalize various subclasses of starlike and convex functions, respectively.
In this paper, sharp bounds  are established for certain Toeplitz determinants constructed over the coefficients and logarithmic coefficients of inverse functions belonging to $\mathcal{S}^*(\varphi)$ and $\mathcal{C}(\varphi)$. Since these classes covers many well-known subclasses, the derived bounds are directly applicable to them as well. 
\end{abstract}
\vspace{0.5cm}

\section{Introduction}
       Let $\mathcal{A}$ denote the class of analytic functions defined on the open unit disk $\mathbb{U}=\{ z \in \mathbb{C}: \vert z \vert < 1 \}$ that have the series expansion
\begin{equation}\label{a2a3a4}
    f(z)= z + \sum_{n=2}^\infty a_n z^n .
\end{equation}
     The subclass of $\mathcal{A}$ consisting of univalent functions is denoted by $\mathcal{S}$. Two of the most studied subclasses of $\mathcal{S}$ are the classes of starlike and convex functions, represented by $\mathcal{S}^*$ and $\mathcal{C}$, respectively. A function $f\in \mathcal{S}^*$ if and only if $\RE (z f'(z)/f(z))>0$. Similarly, a function $f\in \mathcal{C}$ if and only if $\RE (1 + z f''(z)/f'(z))>0 .$ Another significant class is $\mathcal{P}$, which contains analytic functions $p$ in $\mathbb{U}$ that satisfy $\RE p(z)>0$ and have the form
     $$ p(z)= 1+ \sum_{n=1}^\infty p_n z^n.$$
     This class is commonly known as the Carath\'{e}odory class.

      For any two analytic functions $f$ and $g$ in $\mathcal{A}$, the function $f$ is said to be subordinate to $g$,  denoted by $f \prec g$, if there exists a Schwarz function $\omega(z)$ such that $f(z) = g(\omega(z))$.
    The family of all Schwarz functions $\omega(z) $ satisfying $\omega(0)= 0$ and $\vert \omega(z) \vert \leq 1 $ for all $z\in \mathbb{U}$, is expressed by $\mathcal{B}_0$.  Incorporating the notion of subordination, Ma and Minda~\cite{MaMin} unified various subclasses of $\mathcal{S}^*$ and $\mathcal{C}$ by defining
    $$ \mathcal{S}^*(\varphi)= \bigg\{ f \in \mathcal{A} : \frac{z f'(z)}{f(z) } \prec \varphi(z) \bigg\}$$
    and
    $$ \mathcal{C}(\varphi) =\bigg\{ f \in \mathcal{A} : 1 + \frac{z f''(z)}{f'(z) } \prec \varphi(z) \bigg\}  ,$$
    where $\varphi(z)$ is an analytic univalent function in $\mathbb{U}$ that satisfies $\varphi'(0) >0$ and maps the unit disk onto a domain in the right half plane, which is symmetric about the real axis and starlike with respect to $\varphi(0)=1$.  It is obvious that $\mathcal{S}^*((1+z)/(1-z))= \mathcal{S}^*$ and $\mathcal{C}((1+z)/(1-z))= \mathcal{C}$ are the classes of starlike and convex functions, respectively. If $0 \leq \alpha <1$, then $\mathcal{S}^*((1+ (1- 2 \alpha) z)/(1 -z)) =: \mathcal{S}^*(\alpha)$ and $\mathcal{C}((1+ (1- 2 \alpha) z)/(1 -z)) =: \mathcal{C}(\alpha)$ become the classes of starlike and convex functions of order $\alpha$, respectively. For $\varphi(z)=(1+A z)/(1+Bz)$, where $-1 \leq B < A \leq 1$, the classes $\mathcal{S}^*(\varphi)$ and $\mathcal{C}(\varphi)$ reduce to the classes $\mathcal{S}^*[A,B]$ and $\mathcal{C}[A,B]$, respectively, introduced by Janowski~\cite{Jan}. By taking $\phi(z)=((1+z)/(1-z))^\beta$, $0 < \beta \leq 1$, we obtain the classes of strongly starlike and strongly convex functions, denoted by $\mathcal{SS}^*(\beta)$ and $\mathcal{CC}(\beta)$, respectively. Various interesting subclasses of starlike and convex functions by confining the values of $zf'(z)/f(z)$ to a defined region within the right half-plane  were introduced and studied.
     Some of them are listed in  Table \ref{tb1} along with their respective class notations.
\begin{table}[ht]
  \centering
  \caption{Subclasses of starlike functions}
  \label{tb1}
  \begin{tabular}{lll}
    \hline
    {\bf{Class}}  & \textbf{$\varphi(z)$} &  \textbf{Reference} \\ [0.8ex]
    \hline
       $\mathcal{S}^*_{\varrho}$   &   $1+z e^z$   &   Kumar and  Gangania~\cite{KumGan}      \\ [0.8ex]
    \hline
    $\mathcal{S}^*_e$ & $ e^{z}$    & Mendiratta et al.~\cite{MenNagRav} \\ [0.8ex]
    \hline
     $\Delta^*$ & $z + \sqrt{1+z^2}$ & Raina and Sok\'{o}\l~\cite{RaiSok}\\
    \hline
    $\mathcal{S}_P$  & $1 + \frac{2}{\pi^2} \bigg( \log{\frac{1+ \sqrt{z}}{1- \sqrt{z}}} \bigg)^2$ &  R{\o}nning~\cite{Ron} \\ [0.8ex]
  \hline
    $\mathcal{S}^*_L$ & $ \sqrt{1+z^2}$ & Sok\'{o}{\l}  and  Stankiewicz~\cite{SokSta}\\ [0.8ex]
    \hline

  \end{tabular}
\end{table}


    Recently, finding sharp estimates for Toeplitz determinants constructed over the coefficients of functions belonging to different classes gained notable attention of researchers. These determinants have important applications in areas such as algebra, signal processing and time series analysis~\cite{YeLim}. Firstly, Ali et al.~\cite{AliThoVas} considered the Toeplitz determinant for $f(z)=z+\sum_{n=2}^\infty a_n z^n \in \mathcal{A}$, defined by
\begin{equation}\label{Toepla2a3}
     T_{m,n}(f)= \begin{vmatrix}
	a_n & a_{n+1} & \cdots & a_{n+m-1} \\
	{a}_{n+1} & a_n & \cdots & a_{n+m-2}\\
	\vdots & \vdots & \vdots & \vdots\\
     {a}_{n+m-1} & {a}_{n+m-2} & \cdots & a_n\\
	\end{vmatrix}  \quad \quad (m \in \mathbb{N}).
\end{equation}
    They obtained sharp estimates for $\vert T_{2,n}(f) \vert$, $\vert T_{3,1}(f) \vert$ and $\vert T_{3,2}(f) \vert$ for the class $\mathcal{S}$ and some of its subclasses, including $\mathcal{S}^*$ and $\mathcal{C}$. Motivated by this work, several researchers subsequently derived estimates for various Toeplitz determinants associated with different subclasses of $\mathcal{S}$~\cite{AhuKhaRav,CudKwoLecSimSmi,GirKum,GirKum2,GirKum3,LecSimBsm}.

    For function $f\in \mathcal{S}$, consider
    $$ F_f(z) := \log \frac{f(z)}{z}= 2 \sum_{n=1}^\infty \gamma_n z^n, \quad \log{1}=0. $$
    The coefficients $\gamma_n$ associated with each $f\in \mathcal{S}$ are called the logarithmic coefficients. These coefficients play an important role in Milin's conjecture~(\cite[p. 155]{Dur}\cite{Mil}). It is surprising that, the only known sharp estimates for functions in $\mathcal{S}$ are
    $$\vert\gamma_1\vert \leq 1\;\;\; \text{and}\;\;\; \vert \gamma_2 \vert \leq \frac{1}{2}+\frac{1}{e},$$
     whereas finding sharp estimates for other coefficients remains an open problem. Many authors obtained  sharp estimates of the logarithmic coefficients for various subclasses of $\mathcal{S}$~\cite{AliVas,ChoKowKwoLecSim,Zap}.

     The Koebe one-quarter theorem enables us to define inverse function for each $f\in \mathcal{S}$ at least on a disc of radius $1/4$ with the Taylor's series expansion
\begin{equation}\label{bn's}
      f^{-1}(w)= w+ \sum_{n=2}^\infty b_n w^n.
\end{equation}
     By applying the identity $f(f^{-1}(w))=w$, we get
\begin{equation}\label{b2b3b4}
\begin{aligned}
\left.
\begin{array}{ll}
     b_2 &=-a_2,\\
     b_3 &= -a_3 + 2 a_2^2, \\
     b_4 &= -5 a_2^3 + 5 a_2 a_3 -a_4.
\end{array}
\right\}
\end{aligned}
\end{equation}
    Ponnusamy et al.~\cite{PonShaWir} introduced the notion of logarithmic coefficients for the inverse function, defined as
\begin{equation}\label{Gn's}
     F_{f^{-1}}(w) := \log\bigg(\frac{f^{-1}(w)}{w}\bigg) = 2 \sum_{n=1}^\infty \Gamma_n w^n .
\end{equation}
    Applying (\ref{bn's}) and (\ref{b2b3b4}), we obtain
\begin{equation}\label{G2G3G4}
   \Gamma_1 =-\frac{a_2}{2},\;\;  \Gamma_2 = - \frac{1}{2}\bigg( a_3 - \frac{3}{2} a_2^2 \bigg) \;\; \text{and}\;\; \Gamma_3 = -\frac{1}{2} \bigg( a_4 - 4 a_2 a_3 + \frac{10}{3} a_2^3 \bigg).
\end{equation}
   Thus, from (\ref{Toepla2a3}) using the coefficients in (\ref{bn's}), the second order Toeplitz determinants for inverse functions are given by
\begin{equation}\label{Toeplb2b3}
    T_{2,1}(f^{-1})= b_2^2 -b_3^2, \quad T_{2,2}(f^{-1})= b_3^2 -b_4^2.
\end{equation}
   Similarly, using (\ref{Gn's}), the second-order Toeplitz determinants for logarithmic coefficients of inverse functions are given as
\begin{equation}\label{ToeplG2G3}
    T_{2,1}(F_{f^{-1}})=\Gamma_1^2 - \Gamma_2^2, \quad T_{2,2}(F_{f^{-1}})= \Gamma_2^2 - \Gamma_3^2 .
\end{equation}
     Hadi et al.~\cite{HadSalLupAlsAla} derived the sharp bounds for $\vert T_{2,1}(f^{-1})\vert$ and $\vert T_{2,2}(f^{-1})\vert$ for the classes of starlike and convex functions with respect to symmetric points. Mandal et al.~\cite{ManRoyAha} established sharp estimates for second-order Hankel and Toeplitz determinants associated with the logarithmic coefficients of inverse functions for the same classes.  Kumar et al.~\cite{KumTriPan} obtained the bounds for $\vert   T_{2,1}(F_{f^{-1}})\vert$ for a subclass of starlike functions. A good amount of literature is available on these coefficient functionals constructed over the coefficients and logarithmic coefficients of inverse functions for different subclasses of $\mathcal{S}$~\cite{AllSha,GirKum4,GirKum3,SriShaIbrTchKha,WahTumSha}.

     Motivated by these works, we determine the sharp bounds for $\vert T_{2,1}(f^{-1})$, $\vert T_{2,2}(f^{-1})$, $ \vert T_{2,1}(F_{f^{-1}})\vert$ and $ \vert T_{2,2}(F_{f^{-1}})\vert$ for the classes $\mathcal{S}^*(\varphi)$ and $\mathcal{C}(\varphi)$. Special cases of our results  for various subclasses of starlike and convex functions are also presented for specific choices of $\varphi$.

    The following lemmas are required to prove the main results:
\begin{lemma}\cite{ProSzy}\label{Lemma1}
    If $\omega(z) = \sum_{n=1}^\infty c_n z^n \in \mathcal{B}_0$ and $(\sigma , \mu) \in \cup_{i=1}^3 \Omega_i$, then
    $$ \vert c_3 + \sigma c_1 c_2 + \mu c_1^3 \vert \leq \vert \mu \vert ,$$
    where
\begin{align*}
     \Omega_1 &= \bigg\{ (\sigma, \mu) : \vert \sigma \vert \leq 2, \; \mu \geq 1 \bigg\}, \;\;  \Omega_2 = \bigg\{ (\sigma, \mu) : 2 \leq \vert \sigma \vert \leq 4, \; \mu \geq \frac{1}{12} (\sigma^2 + 8) \bigg\},
\end{align*}
   and
   $$  \Omega_3 = \bigg\{ (\sigma, \mu) : \vert \sigma \vert \geq 4, \; \mu \geq \frac{2}{3} (\vert \sigma\vert - 1) \bigg\}. $$
\end{lemma}

\section{For logarithmic inverse coefficients}
   This section is dedicated to establishing estimates for $ \vert T_{2,1}(F_{f^{-1}})\vert$ and $ \vert T_{2,2}(F_{f^{-1}})\vert$ for the classes $\mathcal{S}^*(\varphi)$ and $\mathcal{C}(\varphi)$.
\begin{theorem}\label{thm1}
    Let $f\in \mathcal{S}^*(\varphi)$ and $\varphi(z)= 1+ B_1 z+ B_2 z^2+ B_3 z^3+ \cdots$. If $ \vert B_2 - 2 B_1^2 \vert \geq B_1 $ holds, then
    $$ \vert T_{2,1}(F_{f^{-1}}) \vert  \leq \frac{B_1^2}{4} + \frac{1}{16} (2 B_1^2 - B_2)^2.$$
    The inequality is sharp.
\end{theorem}
\begin{proof}
      Let $f \in \mathcal{S}^*(\varphi)$ be of the form (\ref{a2a3a4}), then by the definition of subordination, there exists a Schwarz function $\omega(z)=\sum_{n=1}^\infty c_n z^n \in \mathcal{B}_0$, such that
    $$ \frac{z f'(z)}{f(z)}= \varphi(\omega(z)).$$
    Using the series expansions of $f$, $\varphi$ and $\omega$, and by comparing the coefficients of same powers of $z$, we get
\begin{equation}\label{a2a3a4Sphi}
\begin{aligned}
\left.
\begin{array}{ll}
    a_2 &= B_1 c_1 \\ \\
    a_3 &= \dfrac{1}{2}((B_1^2 + B_2) c_1^2 +B_1 c_2)  \\ \ \\
    a_4 &= \dfrac{1}{6} ((B_1^3 + 3 B_1 B_2 + 2 B_3 )c_1^3  +  (3 B_1^2 + 4 B_2 )c_1 c_2 + 2 B_1 c_3 ).
\end{array}
\right\}
\end{aligned}
\end{equation}
    In view of the bounds $\vert c_1 \vert\leq 1$, we have
\begin{equation}\label{Sa2}
     \vert a_2 \vert \leq B_1.
\end{equation}
    As a direct consequence of \cite[Theorem 1]{AliRavSee}, we get the following estimate for $f\in \mathcal{S}^*(\varphi)$:
\begin{equation}\label{FSSPHI}
\begin{aligned}
\vert  a_3 - \lambda a_2^2 \vert \leq
\left\{
\begin{array}{ll}
    \dfrac{ B_1^2 + B_2 - 2 \lambda B_1^2 }{2}, & 2 \lambda B_1^2 \leq B_1^2 + B_2 - B_1; \\
    \dfrac{B_1}{2},                      &  B_1^2 + B_2 - B_1 \leq 2 \lambda B_1^2    \leq  B_1^2 + B_2 + B_1 ; \\
    \dfrac{ 2 \lambda B_1^2 - B_1^2 - B_2  }{2} ,  &   2 \lambda B_1^2 \geq B_1^2 + B_2 + B_1.
\end{array}
\right.
\end{aligned}
\end{equation}
   Since the condition $\vert B_2 -2 B_1^2\vert \geq B_1$ holds, it follows that
\begin{equation}\label{Gamma2S}
   \vert \Gamma_2 \vert =  \left\vert a_3 - \frac{3}{2} a_2^2 \right\vert \leq \frac{1}{2}\left\vert B_2 - 2 B_1^2 \right\vert.
\end{equation}
   From~(\ref{G2G3G4}) and (\ref{ToeplG2G3}), we deduce that
\begin{equation}\label{T21fin}
\begin{aligned}
\left.
\begin{array}{ll}
    \vert T_{2,1}(F_{f^{-1}})\vert=  \vert \Gamma_1^2 -\Gamma_2^2 \vert 
    &=  \bigg \vert  \dfrac{a_2^2}{4} - \dfrac{1}{4}\bigg(  a_3 - \dfrac{3}{2} a_2^2 \bigg)^2 \bigg\vert  \\
        &\leq  \dfrac{\vert a_2\vert^2}{4} + \dfrac{1}{4}\bigg\vert  a_3 - \dfrac{3}{2} a_2^2 \bigg\vert^2.
\end{array}
\right\}
\end{aligned}
\end{equation}
    Applying the estimates given in~(\ref{Sa2}) and~(\ref{Gamma2S}) to this inequality, we get the required bound for $\vert T_{2,1}(F_{f^{-1}})$ when $f\in \mathcal{S}^*(\varphi)$.

  Equality case holds for the function $f_{\varphi}(z)\in \mathcal{S}^*(\varphi)$, given by
\begin{equation}\label{ExtSphi}
    f_{\varphi}(z) = z \exp\int_{0}^z \frac{\varphi(i t)-1}{t} dt = z+ i B_1 z^2 - \frac{1}{2}(B_1^2 + B_2) z^2+\cdots.
\end{equation}
   A simple caluclation reveals that for the function $f_\varphi$, we have
   $$\Gamma_1= -\frac{i B_1}{2}, \;\;  \Gamma_2= \frac{B_2 - 2 B_1^2}{4} $$
   and
   $$\left\vert \Gamma_1^2 - \Gamma_2^2 \right\vert= \frac{B_1^2}{4} + \frac{(B_2 - 2 B_1^2)^2}{16},$$
   which shows the sharpness of the bound.
\end{proof}

\begin{theorem}\label{thm2}
    Let $f\in \mathcal{C}(\varphi)$ and $\varphi(z)= 1+ B_1 z+ B_2 z^2+ B_3 z^3+ \cdots$. If $ \vert B_2 - \frac{5}{4} B_1^2\vert \geq B_1 $ holds, then
    $$ \vert T_{2,1}(F_{f^{-1}}/2) \vert \leq \frac{B_1^2}{16} +\frac{1}{144} \bigg( B_2 - \frac{5}{4} B_1^2 \bigg)^2.$$
    The inequality is sharp.
\end{theorem}
\begin{proof}
    Since $f\in \mathcal{C}(\varphi)$, we have
    $$ 1 + \frac{z f''(z)}{f'(z)}= \varphi(\omega(z)),$$
    where $\omega(z)= \sum_{n=1}^\infty\in \mathcal{B}_0$. The series expansions of $f$, $\varphi$ and $\omega$ together with the comparison of same powers of $z$ yield
\begin{equation}\label{a2a3a4Cphi}
\begin{aligned}
\left.
\begin{array}{ll}
    a_2 &= \dfrac{B_1 c_1}{2}, \\ \\
    a_3 &= \dfrac{1}{6}((B_1^2 + B_2) c_1^2 + B_1 c_2),\\ \\
    a_4 &= \dfrac{1}{24} ((B_1^3 + 3 B_1 B_2 + 2 B_3) c_1^3 + (3 B_1^2 + 4 B_2) c_1 c_2 +  2 B_1 c_3 ).
\end{array}
\right\}
\end{aligned}
\end{equation}
  Again by applying the bound $\vert c_1 \vert \leq 1$, we get
\begin{equation}\label{a2Cphi}
   \vert a_2 \vert \leq \frac{B_1}{2}.
\end{equation}
  For $f\in \mathcal{C}(\varphi)$, Ma and Minda~\cite[Theorem 3]{MaMin} established that
\begin{equation}\label{FSCPHI}
\begin{aligned}
\vert a_3 - \lambda a_2^2 \vert \leq
\left\{
\begin{array}{ll}
    \dfrac{( B_2 -\frac{3}{2} \lambda B_1^2 + B_1^2  )}{6}, & 3 \lambda B_1^2 \leq 2 (B_1^2 + B_2 - B_1); \\ \\
    \dfrac{B_1}{6},                      &   2 (B_1^2 + B_2 - B_1) \leq 3 \lambda B_1^2    \leq  2( B_1^2 + B_2 + B_1) ; \\ \\
    \dfrac{(- B_2  + \frac{3}{2} \lambda B_1^2  - B_1^2 )}{6} ,  &   2( B_1^2 + B_2 + B_1) \leq 3 \lambda B_1^2.
\end{array}
\right.
\end{aligned}
\end{equation}
   As by the hypothesis $\vert B_2 - \frac{5}{4} B_1^2\vert \geq B_1$ holds, it follows that
\begin{equation}\label{Gamma2C}
   \vert \Gamma_2\vert = \bigg\vert a_3 - \frac{3}{2} a_2^2 \bigg\vert \leq \frac{1}{6} \bigg\vert B_2 - \frac{5}{4} B_1^2 \bigg\vert.
\end{equation}
   Using the estimates from (\ref{a2Cphi}) and (\ref{Gamma2C}) in (\ref{T21fin}), we obtain
\begin{align*}
    \vert T_{2,1}(F_{f^{-1}}) \vert & \leq  \frac{B_1^2}{16} +\frac{1}{144} \bigg( B_2 - \frac{5}{4} B_1^2 \bigg)^2.
\end{align*}
   A straightforward calculation shows that for the function $h_{\varphi}\in \mathcal{C}(\varphi)$, defined by
\begin{equation}\label{ExtCphi}
    1 + \frac{z h_\varphi''(z)}{h_{\varphi}'(z)} = \varphi(i z),
\end{equation}
  equality case holds, establishing the sharpness of the bound.
\end{proof}

\begin{theorem}
   Let $f\in \mathcal{S}^*(\varphi)$ and \small{$\varphi(z) = 1+ B_1 z+ B_2 z^2 + B_3 z^3 +\cdots.$} If \small{$ \vert B_2 - 2 B_1^2 \vert \geq B_1 $} and $(\sigma_1, \mu_1) \in \cup_{i=1}^3 \Omega_i$ hold, then
   $$ \vert T_{2,2}(F_{f^{-1}}) \vert \leq \frac{(B_2 - 2 B_1^2)^2 }{16}   + \frac{ (9 B_1^3 - 9 B_1  B_2 +2 B_3 )^2 }{144},$$
   where
   $$ \sigma_1= -\frac{(9 B_1^2-4 B_2 )}{2 B_1} \;\;\; \text{and} \;\;\; \mu_1 = \frac{ (9 B_1^3 - 9 B_1  B_2 +2 B_3 )}{2 B_1}. $$
   The bound is sharp.
\end{theorem}
\begin{proof}
    Let $f \in \mathcal{S}^*(\varphi)$ be of the form (\ref{a2a3a4}), then from (\ref{a2a3a4Sphi}), we deduce that
\begin{align*}
  \bigg\vert  -\frac{1}{2} a_4 + 2 a_2 a_3 - \frac{5}{3} a_2^3 \bigg\vert &=\frac{ B_1}{6} \vert c_3  +\sigma_1 c_1 c_2  + \mu_1 c_1^3 \vert,
\end{align*}
   where
    $$ \sigma_1= -\frac{(9 B_1^2-4 B_2 )}{2 B_1} \;\;\; \text{and} \;\;\; \mu_1 = \frac{ (9 B_1^3 - 9 B_1  B_2 +2 B_3 )}{2 B_1}. $$
  Since $(\sigma_1, \mu_1)$ belongs to either $\Omega_1$ or $\Omega_2$ or $\Omega_3$, from Lemma \ref{Lemma1} and (\ref{G2G3G4}), we get
\begin{equation}\label{G3Sphi}
     \vert \Gamma_3 \vert = \bigg\vert  -\frac{1}{2} a_4 + 2 a_2 a_3 - \frac{5}{3} a_2^3 \bigg\vert \leq  \frac{ \vert 9 B_1^3 - 9 B_1  B_2 +2 B_3 \vert }{12}.
\end{equation}
   Using the estimates of $\Gamma_2$ and $\Gamma_3$ from (\ref{Gamma2S}) and (\ref{G3Sphi}), respectively, in (\ref{ToeplG2G3}) together with triangular inequality, we obtain
\begin{align*}
         \vert T_{2,2}(F_{f^{-1}})\vert 
                                          &\leq \vert \Gamma_2 \vert^2 + \vert \Gamma_3 \vert^2  \\
                                         & \leq \frac{1}{16}(B_2 - 2 B_1^2)^2    + \frac{ (9 B_1^3 - 9 B_1  B_2 +2 B_3 )^2 }{144}.
\end{align*}

   It is a simple exercise to check that the sharpness of this bound follows from the function $f_\varphi$ given by (\ref{ExtSphi}).
\end{proof}
\begin{theorem}
   Let $f \in \mathcal{K}(\varphi)$ and $\varphi(z) = 1+ B_1 z+ B_2 z^2 + B_3 z^3 \cdots$. If $ \vert B_2 - \frac{5}{4} B_1^2\vert \geq B_1 $ and $(\sigma_2, \mu_2) \in \cup_{i=2}^3 \Omega_i$ hold, then
\begin{equation}\label{T22Cphi}
     \vert T_{2,2}(F_{f^{-1}}/2) \vert \leq \frac{1}{144} \bigg( B_2 - \frac{5}{4} B_1^2 \bigg)^2 + \frac{ ( 3 B_1^3-5 B_1 B_2+2 B_3 )^2}{2304},
\end{equation}
    where
    $$ \sigma_2 = -\frac{ (5 B_1^2-4 B_2 )}{2 B_1} \;\; \text{and} \;\; \mu_2  =\frac{ (3 B_1^3-5 B_1 B_2+2 B_3)}{2 B_1}.$$
    The estimate is sharp.
\end{theorem}
\begin{proof}
    For $f(z)=z + \sum_{n=2}^\infty a_n z^n \in \mathcal{K}(\varphi)$, applying the expressions of $a_2$, $a_3$ and $a_4$ from (\ref{a2a3a4Cphi}) yields
\begin{align*}
 \bigg\vert -\frac{1}{2} a_4 + 2 a_2 a_3 - \frac{5}{3} a_2^3\bigg\vert &= \frac{ B_1}{24}  \big\vert c_3 + \sigma_2 c_1 c_2 + \mu_2 c_1^3  \big\vert,
\end{align*}
    where
     $$ \sigma_2 =-\frac{ (5 B_1^2-4 B_2 )}{2 B_1} \;\; \text{and} \;\; \mu_2  =\frac{ (3 B_1^3-5 B_1 B_2+2 B_3)}{2 B_1}. $$
     From Lemma \ref{Lemma1} and (\ref{G2G3G4}), we get
\begin{equation}\label{G3Cphi}
   \vert \Gamma_3 \vert = \bigg\vert -\frac{1}{2} a_4 + 2 a_2 a_3 - \frac{5}{3} a_2^3\bigg\vert \leq \frac{ \vert 3 B_1^3-5 B_1 B_2+2 B_3 \vert}{48}.
\end{equation}
   Using the estimates of $\vert\Gamma_2\vert$ and $\vert \Gamma_3 \vert$ from (\ref{Gamma2C}) and (\ref{G3Cphi}), respectively, in (\ref{ToeplG2G3}), we obtain
\begin{align*}
     \vert T_{2,2}(F_{f^{-1}}) \vert &\leq \vert \Gamma_2 \vert^2 + \vert \Gamma_3 \vert^2  \\
                                         &\leq  \frac{1}{144} \bigg( B_2 - \frac{5}{4} B_1^2 \bigg)^2 + \frac{ ( 3 B_1^3-5 B_1 B_2+2 B_3 )^2}{2304}.
\end{align*}
    For the function $h_\varphi\in \mathcal{C}(\varphi)$ defined in (\ref{ExtCphi}), we have
     $$ \Gamma_2 =\frac{(4 B_2-5 B_1^2)}{48} , \;\; \Gamma_3 = \frac{i (3 B_1^3 - 5 B_1  B_2 +2 B_3 )}{48}. $$
   A direct calculation shows that the equality in (\ref{T22Cphi}) is attained for the function $h_\varphi$, showing the sharpness of the bound.
\end{proof}

\section{For coefficients of inverse function}
    In this section, we determine the sharp bounds for the second-order Toeplitz determinants of the coefficients of inverse functions, namely $\vert T_{2,1}(f^{-1})\vert$ and  $ \vert T_{2,2}(f^{-1})\vert$, in the classes $\mathcal{S}^*(\varphi)$ and $\mathcal{C}(\varphi)$.
\begin{theorem}
   Let $f\in \mathcal{S}^*(\varphi)$ and $\varphi(z)= 1 +B_1 z +B_2 z^2 +\cdots$. If $B_1 \leq \vert 3 B_1^2 - B_2 \vert$ holds, then
   $$ \vert T_{2,1}(f^{-1})\vert \leq B_1^2 + \frac{(3 B_1^2 - B_2)^2}{4}. $$
   The estimate is sharp.
\end{theorem}
\begin{proof}
    Let $f\in \mathcal{S}^*(\varphi)$ be of the form (\ref{a2a3a4}). Using the estimate of $\vert a_2\vert$ from (\ref{Sa2}) in (\ref{b2b3b4}), we obtain
\begin{equation}\label{b2Sphi}
    \vert b_2 \vert \leq B_1.
\end{equation}
   Since $\vert 3 B_1^2 - B_2 \vert \geq B_1$, by (\ref{b2b3b4}) and (\ref{FSSPHI}), it follows that
\begin{equation}\label{b3Sphi}
    \vert b_3\vert = \vert a_3 - 2 a_2^2 \vert \leq \frac{\vert 3 B_1^2 - B_2 \vert}{2}.
\end{equation}
   Using the bounds from (\ref{b2Sphi}) and (\ref{b3Sphi}) in (\ref{Toeplb2b3}), we deduce that
\begin{align*}
    \vert T_{2,1}(f^{-1})\vert \leq \vert b_2 \vert^2 +\vert b_3 \vert^2 \leq  B_1^2 +\frac{(3 B_1^2 - B_2)^2}{4}.
\end{align*}
     This inequality becomes equality for the function $f_\varphi\in \mathcal{S}^*(\varphi)$ given by (\ref{ExtSphi}), as for $f_\varphi$, we have
     $$  b_2 = -i B_1,  \;\; b_3 = \frac{B_2 - 3B_1^2}{2}, $$
    which shows that the bound is sharp and completes the proof.
\end{proof}
\begin{theorem}
   Let $f\in \mathcal{C}(\varphi)$ and $\varphi(z)=1+ B_1 z +B_2 z^2 +\cdots$. If $B_1 \leq \vert 2 B_1^2 - B_2 \vert$ holds, then
   $$ \vert T_{2,1}(f^{-1})\vert \leq \frac{B_1^2}{4}+ \frac{(2B_1^2 - B_2)^2}{36}. $$
   The estimate is sharp.
\end{theorem}
\begin{proof}
   For $f\in \mathcal{C}(\varphi)$, estimate of $\vert a_2 \vert$ given in (\ref{a2Cphi}) together with (\ref{b2b3b4}) yields
\begin{equation}\label{b2Cphi}
    \vert b_2 \vert \leq \frac{B_1}{2}.
\end{equation}
   In view of the hypothesis $\vert 2 B_1^2 - B_2 \vert \geq B_1$, Fekete-Szeg\"{o} estimate for $f\in \mathcal{C}(\varphi)$ given in (\ref{FSCPHI}), gives
\begin{equation}\label{b3Cphi}
    \vert b_3 \vert = \vert a_3 -2 a_2^2 \vert \leq \frac{\vert 2 B_1^2 - B_2 \vert}{6}.
\end{equation}
   From these estimates of $\vert b_2\vert$ and $\vert b_3\vert$ given in (\ref{b2Cphi}) and (\ref{b3Cphi}), respectively, we obtain
   $$ \vert T_{2,1}(f^{-1})\vert \leq \vert b_2 \vert^2 +\vert b_3 \vert^2 \leq \frac{B_1^2}{4}+ \frac{( 2 B_1^2 - B_2 )^2}{36}. $$
   It can be easily seen that the bound is sharp for $h_\varphi\in \mathcal{C}(\varphi)$ given by (\ref{ExtCphi}).
\end{proof}
\begin{theorem}
   Let $f\in \mathcal{S}^*(\varphi)$ and $\varphi(z) = 1 + B_1 z+ B_2 z^2  + \cdots$. If $(\sigma_3, \mu_3) \in \cup_{i=2}^3 \Omega_i$ and $B_1 \leq \vert 3 B_1^2 - B_2 \vert$ hold, then
   $$ \vert T_{2,2}(f^{-1}) \vert \leq  \frac{(B_2 - 3 B_1^2)^2}{4}+ \frac{ (8 B_1^3 - 6 B_1 B_2 + B_3)^2}{9}, $$
   where
   $$ \sigma_3 = \frac{2 (B_2-3 B_1^2 )}{B_1} \;\; \text{and} \;\; \mu_3 = \frac{(8 B1^3 - 6 B1 B2 + B3)}{B_1}. $$
   The estimate is sharp.
\end{theorem}
\begin{proof}
   For $f\in \mathcal{S}^*(\varphi)$, the expressions of $a_2$, $a_3$ and $a_4$ from (\ref{a2a3a4Sphi}) yield
   $$ \vert  -5 a_2^3 + 5 a_2 a_3 -a_4 \vert = \frac{B_1}{3} \vert c_3 + \sigma_3 c_1 c_2  + \mu_3 c_1^3 \vert, $$
   where
   $ \sigma_3 = {2 (B_2-3 B_1^2 )}/{B_1}$  and  $\mu_3 = {(8 B_1^3 - 6 B_1 B_2 + B_3)}/{B_1}. $
   By applying Lemma~\ref{Lemma1} together with equation~(\ref{b2b3b4}), we obtain
\begin{equation}\label{b4Sphi}
   \vert b_4 \vert  \leq \frac{\vert 8 B_1^3 - 6 B_1 B_2 + B_3 \vert}{3}.
\end{equation}
   From (\ref{Toeplb2b3}), it follows that
   $$ \vert  T_{2,2}(f^{-1}) \vert \leq \vert b_3\vert^2 + \vert b_4 \vert^2. $$
   Using the bounds for $\vert b_3\vert$ and $\vert b_4\vert$ provided in (\ref{b3Sphi}) and (\ref{b4Sphi}), respectively, the inequality leads to the desired result.

   For the function $f_\varphi\in \mathcal{S}^*(\varphi)$ given by (\ref{ExtSphi}), we have
   $$ b_3 =  \frac{B_2 - 3 B_1^2}{2}, \;\; b_4= \frac{ i (8 B_1^3 - 6 B_1 B_2 + B_3)}{3} $$
   and
   $$ \vert b_3^2 - b_4^2 \vert = \frac{(B_2 - 3 B_1^2)^2}{4}+ \frac{ (8 B_1^3 - 6 B_1 B_2 + B_3)^2}{9}, $$
   which demonstrates the sharpness of the bound.
\end{proof}
\begin{theorem}\label{lastthm}
   Let $f\in \mathcal{C}(\varphi)$ and $\varphi(z) = 1 + B_1 z+ B_2 z^2 + \cdots$. If $(\sigma_4, \mu_4) \in \cup_{i=1}^3 \Omega_i$ and $B_1 \leq \vert 2 B_1^2 - B_2 \vert$ hold, then
   $$ \vert b_3^2 - b_4^2 \vert \leq \frac{(B_2-2 B_1^2)^2}{36}+  \frac{(6 B_1^3 - 7 B_1 B_2 + 2 B_3)^2}{576}, $$
   where
   $$ \sigma_4 = \frac{ 4 B_2 -7 B_1^2}{ 2 B_1}, \;\; \mu_4 = \frac{6 B_1^3 - 7 B_1 B_2 + 2 B_3}{2 B_1}. $$
   The estimate is sharp.
\end{theorem}
\begin{proof}
     Substituting the values of  $a_2$, $a_3$ and $a_4$ from (\ref{a2a3a4Cphi}) for $f\in \mathcal{C}(\varphi)$ , we get
     $$ \vert  -5 a_2^3 + 5 a_2 a_3 -a_4 \vert  = \frac{B_1 }{12}\vert c_3 +\sigma_4 c_1 c_2   + \mu_4  c_1^3  \vert,  $$
     where
     $ \sigma_4 = (4 B_2 -7 B_1^2)/ (2 B_1)$ and $\mu_4 = (6 B_1^3 - 7 B_1 B_2 + 2 B_3)/(2 B_1). $
     Applying Lemma \ref{Lemma1} together with (\ref{b2b3b4}), we obtain
\begin{equation}\label{b4Cphi}
     \vert b_4 \vert   \leq \frac{6 B_1^3 - 7 B_1 B_2 + 2 B_3}{24}.
\end{equation}
    From (\ref{Toeplb2b3}), we deduce that
     $$ \vert  T_{2,2}(f^{-1}) \vert \leq \vert b_2\vert^2 + \vert b_3 \vert^2. $$
     By substituting the bounds for $\vert b_3\vert $ and $\vert b_4\vert $ given by (\ref{b3Cphi}) and (\ref{b4Cphi}), respectively, we arrive at the required result.

     Sharpness of the bound follows from $h_\varphi\in \mathcal{C}(\varphi)$ defined by (\ref{ExtCphi}) as for this function, we have
     $$ b_3= \frac{(B_2-2 B_1^2)}{6},\quad b_4=  \frac{i (6 B_1^3 - 7 B_1 B_2 + 2 B_3)}{24} $$
     and
     $$\vert b_3^2 - b_4^2 \vert = \frac{(B_2-2 B_1^2)^2}{36}+  \frac{(6 B_1^3 - 7 B_1 B_2 + 2 B_3)^2}{576}, $$
     which completes the proof.
\end{proof}
\section{Special Cases }
    The classes $\mathcal{S}^*(\varphi)$ and $\mathcal{C}(\varphi)$  comprise several well known subclasses including $\mathcal{S}^*[A,B]$, $\mathcal{C}[A,B]$, $\mathcal{S}^*(\alpha)$, $\mathcal{C}(\alpha)$, $\mathcal{S}^*$, $\mathcal{C}$, $\mathcal{SS}^*(\beta)$, $\mathcal{CC}(\beta)$ and those presented in Table \ref{tb1}.  Theorem \ref{thm1}--\ref{lastthm} provide the sharp estimates of $\vert T_{2,1}(F_{f^{-1}})\vert $, $\vert T_{2,2}(F_{f^{-1}})\vert $, $\vert T_{2,1}({f^{-1}})\vert$ and $\vert T_{2,2}({f^{-1}})\vert$ for these subclasses of starlike and convex functions, corresponding to appropriate choices of the function $\varphi$.
\begin{corollary}
   Let $f\in \mathcal{S}^*[A,B]$.Then the following sharp bounds hold:
\begin{enumerate}
  \item If $\vert 2 A-B \vert \geq 1$, then
  $$ \vert T_{2,1}(F_{f^{-1}})\vert \leq \frac{1}{16} (A-B)^2 (4 A^2-4 A B+B^2+4 ).$$
  \item If $\vert 2 A-B \vert \geq 1$ and $(\sigma_1, \mu_1) = \cup_{i=1}^3 \Omega_i$, then
  $$ \vert T_{2,2}(F_{f^{-1}})\vert \leq   \frac{(A - B )^2 ((9 A^2 - 9 A B + 2 B^2 )^2 + 9 (B-2 A)^2 )}{144}, $$
  where
  $ \sigma_1= (5 B-9 A)/{2}$ and $\mu_1= (3 A-2 B) (3 A-B)/2. $
  \item If $\vert 3 A-2 B \vert \geq 1$, then
  $$ \vert T_{2,1}({f^{-1}})\vert \leq \frac{ (A-B)^2((3 A-2 B)^2+4 )  }{4}.$$
  \item If $\vert 3 A-2 B \vert \geq 1$ and $(\sigma_3, \mu_3) = \cup_{i=2}^3 \Omega_i$, then
  $$ \vert T_{2,2}({f^{-1}})\vert \leq  \frac{ (A-B)^2 (9 (3 A-2 B)^2+4 (B-2 A)^2 (4 A-3 B)^2 )}{36}, $$
  where $ \sigma_3 =4 B-6 A$ and $\mu_3= 8 A^2-10 A B+3 B^2. $
\end{enumerate}
\end{corollary}
\begin{corollary}
   Let $f\in \mathcal{C}[A,B]$. Then the following sharp bounds hold:
\begin{enumerate}
  \item If $\vert 5 A-B\vert \geq 4$, then
  $$ \vert T_{2,1}(F_{f^{-1}})\vert \leq \frac{(A-B)^2 (25 A^2-10 A B + B^2 + 144 )}{2304},$$
  \item If $\vert 5 A-B\vert \geq 4$ and $(\sigma_2, \mu_2)\in \cup_{i=2}^3 \Omega_i$, then
  $$  \vert T_{2,2}(F_{f^{-1}})\vert \leq \frac{(A-B)^2 (A^2 (B-3 A)^2+(B-5 A)^2)}{2304}, $$
  where
  $\sigma_2 =(B-5 A)/2$ and $\mu_2 =  A (3 A-B)/2.$
  \item If $\vert 2 A-B\vert \geq A-B$, then
  $$     \vert T_{2,1}({f^{-1}})\vert \leq  \frac{(A-B)^2 ((B-2 A)^2+9 )}{36}.   $$
  \item If $\vert 2 A-B\vert \geq A-B$ and $(\sigma_4, \mu_4)\in \cup_{i=1}^3 \Omega_i$, then
  $$  \vert T_{2,2}({f^{-1}})\vert  \leq  \frac{(2 A^2-3 A B+B^2 )^2 ((B-3 A)^2+16 )}{576}, $$
  where
  $\sigma_4 =(7 A-3 B)/{2} $ and $\mu_4 = (6 A^2-5 A B+B^2)/2.$
\end{enumerate}
\end{corollary}
\begin{corollary}
    Let $f \in \mathcal{S}^*(\alpha)$. Then the following hold:
\begin{enumerate}
  \item If $\alpha \in [0,1/2]$, then
   $$ \vert T_{2,1}(F_{f^{-1}})\vert \leq \frac{(1- \alpha)^2((3-4 \alpha )^2 + 4 ) }{4}. $$
   \item If $\alpha \in [0,7/15]$, then
    $$ \vert T_{2,2}(F_{f^{-1}})\vert \leq \frac{(1- \alpha )^2 (9 (3-4 \alpha )^2+4 (2-3 \alpha )^2 (5-6 \alpha )^2)  }{36} . $$
  \item If $\alpha \in [0,2/3]$, then
    $$  \vert T_{2,1}({f^{-1}})\vert  \leq   (1 - \alpha)^2 \left(36 \alpha ^2-60 \alpha +29\right). $$
  \item If $\alpha \in [0,3/5]$, then
    $$  \vert T_{2,2}({f^{-1}})\vert  \leq \frac{(1- \alpha)^2 (9 (5-6 \alpha )^2+4 (3-4 \alpha )^2 (7-8 \alpha )^2) } {9}  . $$
\end{enumerate}

\end{corollary}

\begin{corollary}
    If $f\in \mathcal{C}(\alpha)$, Then the following hold:
\begin{enumerate}
  \item If $\alpha \in [0,1/5]$, then
    $$ \vert T_{2,1}(F_{f^{-1}})\vert \leq \frac{5}{144} (1- \alpha )^2 (5 \alpha ^2-6 \alpha +9 ). $$
  \item If $\alpha \in [0,7/47]$, then
    $$ \vert T_{2,2}(F_{f^{-1}})\vert \leq \frac{(\alpha -1)^2 ((6 \alpha ^2-7 \alpha +2)^2+(3-5 \alpha )^2)}{144}. $$
  \item If $\alpha \in [0,1/2]$, then
    $$  \vert T_{2,1}({f^{-1}})\vert  \leq \frac{((1 -\alpha )^2(3-4 \alpha )^2+9)}{9}  . $$
  \item If  $\alpha \in [0,39/95]$, then
    $$  \vert T_{2,2}({f^{-1}})\vert  \leq  \frac{( 1 -\alpha)^2 ((2-3 \alpha )^2+4 ) (3-4 \alpha )^2 } {36} . $$
\end{enumerate}
\end{corollary}

\begin{corollary}
    If $f\in \mathcal{SS}^*(\beta)$, Then the following hold:
\begin{enumerate}
  \item If $\beta \in [1/3,1] $, then
    $$  \vert T_{2,1}(F_{f^{-1}})\vert \leq \frac{\beta^2 (9 \beta^2 + 4 )}{4} \;\; \text{and} \;\; \vert T_{2,2}(F_{f^{-1}})\vert \leq \frac{\beta^2 (3364 \beta^4+961 \beta ^2 + 4 )}{324}  .  $$
  \item If  $\beta \in [1/5,1]$, then
    $$  \vert T_{2,1}({f^{-1}})\vert  \leq   \beta ^2 (25 \beta^2 + 4 )\;\; \text{and}\;\;  \vert T_{2,2}({f^{-1}})\vert  \leq  \frac{1}{81} \beta^2 (15376 \beta^4 + 2521 \beta^2 + 4 ). $$
\end{enumerate}
\end{corollary}

\begin{corollary}
    Let $f\in \mathcal{CC}(\beta)$. Then the following hold:
\begin{enumerate}
  \item If $\beta \in [2/3,1] $, then
    $$  \vert T_{2,1}(F_{f^{-1}})\vert \leq  \frac{\beta^2 (\beta^2+4 )}{16} \;\; \quad{and}\;\; \vert T_{2,2}(F_{f^{-1}})\vert \leq  \frac{\beta^2 (25 \beta^4 + 91 \beta^2 + 1 )}{1296} .  $$
  \item If $\beta \in [1/3,1]$, then
    $$  \vert T_{2,1}({f^{-1}})\vert  \leq  \beta^2 (\beta^2 + 1 ).  $$
  \item If $\beta \in [\sqrt{2/17},1]$, then
    $$  \vert T_{2,2}({f^{-1}})\vert  \leq  \frac{\beta^2 (289 \beta^4+358 \beta^2+1 )} {324}.  $$
\end{enumerate}
\end{corollary}

\begin{corollary}
    Let $f \in \mathcal{S}^*$. Then the following sharp bounds hold:
   $$  \vert T_{2,1}(F_{f^{-1}})\vert \leq  \frac{13}{4}, \;\; \vert T_{2,2}(F_{f^{-1}})\vert \leq \frac{481}{36}, \;\; \vert T_{2,1}({f^{-1}})\vert \leq 29, \;\;  \vert T_{2,2}({f^{-1}})\vert \leq  221.  $$
\end{corollary}

\begin{corollary}
    Let $f \in \mathcal{C}$. Then the following sharp bounds hold:
   $$  \vert T_{2,1}(F_{f^{-1}})\vert \leq  \frac{5}{16}, \;\; \vert T_{2,2}(F_{f^{-1}})\vert \leq \frac{13}{144}, \;\; \vert T_{2,1}({f^{-1}})\vert \leq 2 , \;\;  \vert T_{2,2}({f^{-1}})\vert \leq  2.  $$
\end{corollary}

\begin{corollary}
    If $f \in \mathcal{S}^*_\varrho$, then the following sharp estimates hold:
    $$ \vert T_{2,1}(F_{f^{-1}})\vert \leq \frac{5}{16}, \quad \vert T_{2,1}({f^{-1}})\vert \leq 2 \quad \text{and} \quad   \vert T_{2,2}({f^{-1}})\vert \leq \frac{61}{36}. $$
\end{corollary}
\begin{corollary}
    If $f \in \mathcal{S}_P$, then the following sharp estimates hold:
    $$ \vert T_{2,1}({f^{-1}})  \leq  \frac{128  (648-36 \pi ^2+5 \pi ^4 )}{9 \pi ^8} $$
    and
    $$ \vert T_{2,2}({f^{-1}}) \leq \frac{64 (\pi ^2-36)^2}{9 \pi ^8}+\frac{64 (23040-1440 \pi ^2+23 \pi ^4)^2}{18225 \pi ^{12}}. $$
\end{corollary}
\begin{corollary}
  If $f\in \mathcal{S}^*_e$, then the following sharp estimates hold:
  $$ \vert T_{2,1}(F_{f^{-1}})\vert \leq \frac{25}{64}, \;\; \vert T_{2,2}(F_{f^{-1}})\vert \leq \frac{785}{2592}, \;\; \vert T_{2,1}({f^{-1}})\vert \leq \frac{41}{16},\;\;\vert T_{2,2}({f^{-1}})\vert \leq \frac{5869}{1296}. $$
\end{corollary}
\begin{corollary}
  If $f\in \Delta^*$, then the following sharp estimates hold:
  $$ \vert T_{2,1}(F_{f^{-1}})\vert \leq  \frac{25}{64}, \;\; \vert T_{2,2}(F_{f^{-1}})\vert \leq \frac{9}{32}, \;\; \vert T_{2,1}({f^{-1}})\vert \leq \frac{41}{16}, \;\;  \vert T_{2,2}({f^{-1}})\vert \leq    \frac{625}{144}.$$
\end{corollary}
\begin{corollary}
    If $f\in \mathcal{S}^*_L$, then the following sharp estimates hold:
  $$ \vert T_{2,1}(F_{f^{-1}})\vert \leq  \frac{1}{64}\;\; \text{and} \;\; \vert T_{2,1}({f^{-1}})\vert \leq  \frac{1}{16} .$$
\end{corollary}
\section*{Declarations}

\subsection*{Conflict of interest}
	The authors declare that they have no conflict of interest.
\subsection*{Data Availability} Not Applicable.
\noindent

\end{document}